\documentclass{amsart}

\usepackage{amsmath,amsthm,amssymb,amscd,amstext,amsopn,amsxtra,amsfonts, eucal, paralist, enumitem, fancyhdr, mathtools, tikz-cd, rotating}

\usepackage[all]{xy}
\SelectTips{eu}{}
\entrymodifiers={+!!<0pt,\fontdimen22\textfont2>}

\usepackage{url}
\usepackage{cite}
\usepackage[hidelinks]{hyperref}

\newcounter{gentheorem} \numberwithin{gentheorem}{section}
\theoremstyle{plain}
\newtheorem{defn}[gentheorem]{Definition}
\newtheorem{thm}[gentheorem]{Theorem}
\newtheorem{lemma}[gentheorem]{Lemma}
\newtheorem{prop}[gentheorem]{Proposition}

\theoremstyle{remark}

\newtheorem{remark}[gentheorem]{Remark}

\theoremstyle{definition}

\numberwithin{equation}{section}

\def\cC{\mathcal{C}}

\def\cO{\mathcal{O}}

\def\11{\mathbf{1}}

\def\CC{\mathbb{C}}

\def\FF{\mathbb{F}}

\def\QQ{\mathbb{Q}}

\def\ZZ{\mathbb{Z}} 

\def\fa{\mathfrak{a}}

\def\fP{\mathfrak{P}}

\def\End{\mathrm{End}}

\def\Gal{\mathrm{Gal}}

\def\Hom{\mathrm{Hom}}

\def\Jac{\mathrm{Jac}}

\newcommand{\altdef}[2]{\Isom{#1}{#2}^{\text{Fr}}}
\newcommand{\Pic}[1]{\mathcal{C}^{#1\Phi}}
\newcommand{\Picr}[1]{\widetilde{\Pic{#1}}}
\newcommand{\Isom}[2]{\text{CM}_{#1,#2}}
\newcommand{\Isomr}[2]{\widetilde{\text{CM}}_{#1,#2}}
\newcommand{\Cl}{\operatorname{Cl}}
\usepackage{url}
	
\title[Constructing Picard Curves]{Constructing Picard Curves with Complex Multiplication using the Chinese Remainder Theorem}	  	  
\author{Sonny Arora}
\address{Department of Mathematics \\ The Pennsylvania State
	University \\ University, Park, PA 16802, USA}
\email{sza149@psu.edu}

\author{Kirsten Eisentr\"ager} 
\thanks{The first author was partially supported by National
	Science Foundation grants DMS-1056703 and CNS-1617802. The second author was partially
	supported by National Science Foundation awards DMS-1056703 and
	CNS-1617802, and by the National Security Agency (NSA) under Army
	Research Office (ARO) contract number W911NF-12-1-0541.}
\address{Department of Mathematics \\ The Pennsylvania State
	University \\ University, Park, PA 16802, USA}
\email{eisentra@math.psu.edu}

\begin{document}

\begin{abstract}
  We give a new algorithm for constructing Picard curves over a finite
  field with a given endomorphism ring. This has important
  applications in cryptography since curves of genus 3 allow one to
  work over smaller fields than the elliptic curve case. For a sextic
  CM-field $K$ containing the cube roots of unity, we define and
  compute certain class polynomials modulo small primes and then use
  the Chinese Remainder Theorem to construct the class polynomials
  over the rationals. We also give some examples.
\end{abstract}

\maketitle


\section{Introduction}
For cryptographic protocols whose security relies on the difficulty of
the discrete log problem, one often wants to find a group whose order
is divisible by a large prime. One option for the group is the
group of points of an elliptic curve over a finite field, or more
generally, the group of points on the Jacobian of a curve over a
finite field. Thus, we are interested in the problem of finding curves
over finite fields whose Jacobian has a given number of points.

For elliptic curves, Atkin and Morain showed in \cite{AM93} that one
can use the theory of complex multiplication to solve this
problem. The approach taken in \cite{AM93} involves computing the
Hilbert class polynomial with respect to an imaginary quadratic field
by evaluating modular $j$-invariants at certain values. An alternate
method to construct the Hilbert class polynomial, taken in
\cite{CNST98} and \cite{ALV04}, is to compute the polynomial modulo
several small primes and then reconstruct the polynomial using the
Chinese Remainder Theorem. In the genus 2 case, analogous to the
construction of the Hilbert class polynomial, one wishes to construct
the so-called Igusa class polynomials. In this case, one can again use
a Chinese Remainder Theorem approach to construct the Igusa class
polynomials as shown in \cite{EL09}, \cite{FL08}.
	
If one wishes to construct genus 3 curves with a given number of
points, less is known. Genus 3 curves fall into two classes:
hyperelliptic curves and non-hyperelliptic plane quartics. One
difficulty in the case of genus 3 curves is that there is no theory of
invariants which works for all genus 3 curves. However, invariants do
exist for the classes of hyperelliptic curves and non-hyperelliptic
plane quartics separately. By making restrictions on the type of genus
3 curves considered, algorithms for constructing genus 3 curves with
complex multiplication have been presented in \cite{Wen01},
\cite{KW05}, \cite{LS16}, \cite{BILV16}, and \cite{KLLRSS17}. All
these papers take a complex analytic approach to constructing genus 3
curves similar to the method in \cite{AM93}. The papers \cite{Wen01},
\cite{BILV16} deal with constructing hyperelliptic genus 3 curves with
complex multiplication. The paper \cite{KW05} and its improvement
\cite{LS16} deal with constructing Picard curves with complex
multiplication, while \cite{KLLRSS17} deals with
constructing plane quartics defined over $\QQ$ with complex
multiplication.  Due to the numerous improvements to the Chinese
Remainder theorem approach in the elliptic curve case \cite{BBEL08},
\cite{Sut12}, it is of interest to try to implement a Chinese
Remainder Theorem approach for the construction of genus 3
curves. This is the aim of this paper.
	
As in \cite{KW05}, we will restrict our attention to Picard
curves. These are genus 3 curves of the form $y^3 = f(x)$ where
$\deg(f) = 4$ and $f$ has no repeated roots over the algebraic
closure. One advantage to using these curves is that it is very simple
to generate representatives for all isomorphism classes of Picard curves over a finite
field. Also, if $K$ is a sextic CM-field that contains the cube
roots of unity, then, by \cite [Lemma 1]{KW05}, all simple, principally
polarized abelian varieties of dimension 3 with complex multiplication
by $\cO_K$ arise as the Jacobians of Picard curves, so we can use
Picard curves in a CRT approach. 

\subsection{Statement of theorem} Let $K$ be a sextic CM-field
containing the cube roots of unity. Fix a primitive CM-type $\Phi$ on
the field $K$. Our first step
will be to define suitable class polynomials for $(K, \Phi)$. For
this we will require invariants for Picard curves.

We work with the set of invariants for Picard curves $j_1, j_2, j_3$
defined in \cite{KLS18}. They are discussed in more detail in
Section~\ref{sec:invariants}.

We now wish to introduce class polynomials for Picard curves. Recall,
the Hilbert class polynomial for an imaginary quadratic field $K$ has
as roots the $j$-invariants of elliptic curves with complex
multiplication by the full ring of integers $\cO_K$ of $K$. Analogous
to this situation, we would like the class polynomials we define, for
a sextic CM-field $K$ containing the cube roots of unity, to have as
roots the invariants of Picard curves with complex multiplication by
$\cO_K$. A complication that does not arise in the genus 1 case is
that we will need to restrict to Picard curves whose Jacobian has a
given primitive CM-type on $K$. In genus 2, a restriction on the
CM-type for class polynomials was discussed in \cite{LR13}.

We would like our class polynomials to be defined over $\QQ$. This
will allow us to multiply by a large enough integer to clear
denominators and hence use the Chinese remainder theorem on the
resulting polynomials modulo various primes. For an abelian variety
$A$ of CM-type $(K, \Phi)$ and for $\sigma \in \Gal(\overline{\QQ}/\QQ)$, $A^\sigma$ is of type $(K,
\sigma\Phi)$. Thus, we define class polynomials for $i=1,...,3$ as
follows:
\[H_i^\Phi := \prod (X - j_i(C)),\] where the product runs over all
isomorphism classes of Picard curves $C/\CC$ whose Jacobian has complex
multiplication by $\cO_K$ of type $\sigma \Phi$ for some
$\sigma \in \Gal(\overline{\QQ}/\QQ)$. These polynomials will be
defined over $\QQ$. Should one want to re-construct a Picard curve
$C/\CC$ such that $\End(\Jac(C)) \cong \cO_K$ from the roots of the
class polynomials, it is more convenient to work with a different set
of class polynomials, introduced in \cite{GHKRW06} in the genus 2
setting. This is discussed more in Section~\ref{sec:red-of-class}.

We have the following theorem:

\begin{thm} The following algorithm takes as input a sextic CM-field
	$K$ containing the cube roots of unity and a primitive CM-type
	$\Phi$ on $K$. Assuming the bound $B$ in Theorem~\ref{thm:CRT} is known, the algorithm outputs the class polynomials
	$H_i^\Phi$, where $i=1,...,3$, corresponding to the type $(K,\Phi)$
	\begin{enumerate}
		\item Construct a set of rational primes $S$ which satisfy
		\begin{enumerate}
			\item $2 \not \in S$
			\item Each $p \in S$ splits completely in $K$
			\item Each $p \in S$ splits completely into principal ideals in $K^*$, the reflex field for the type $(K, \Phi)$.
			\item $\prod_{p \in S}p > B$ where $B$ is the bound in
			Theorem~\ref{thm:CRT}. 
		\end{enumerate}
		\item Form the class polynomials $H_i^\Phi$ modulo $p$ for every $p \in S$. Let $H_{i,p} := H_i^\Phi \mod p$. Then
		\[H_{i,p} =  \prod(X - j_i(C)),\]
		where the product is over all $\overline{\FF}_p$-isomorphism classes
		of Picard curves that arise as the reduction of a Picard curve over
		$\CC$ whose Jacobian has complex multiplication by $\cO_K$ of type
		$\sigma\Phi$ for some $\sigma \in \Gal(\overline{\QQ}/\QQ)$.
		\item Form the polynomials $H_i^\Phi$ from the $H_{i,p}$, $p \in S$, using the Chinese Remainder Theorem. 
	\end{enumerate}
	\label{thm:main-algorithm}
\end{thm}

We review background from the theory of complex
multiplication in Section~\ref{sec:results-cm} and prove some results
we will need. In Section~\ref{sec:invariants} we review
invariants of Picard curves. In Section~\ref{sec:red-of-class}, we
discuss reducing class polynomials modulo primes. In
Section~\ref{sec:isomorphism-classes} we show how to compute
$H_i^\Phi$ modulo a prime $p$ and we prove
Theorem~\ref{thm:main-algorithm}. Section~\ref{sec:endo-comp}
discusses the endomorphism ring computation, and in
Section~\ref{sec:examples} we give some examples.

\section{Results from complex multiplication}
\label{sec:results-cm}
\begin{defn}[CM-type] Let $K$ be a CM-field of degree $2g$ and let $\Omega$ be an
  algebraically closed field of characteristic 0. Denote by
  $\Hom(K, \Omega) = \{\phi_1, \phi_2, ..., \phi_{2g}\}$ the set of
  embeddings of $K$ into $\Omega$. Furthermore, let $\rho$ denote the
  automorphism inducing complex conjugation on $K$. Then any subset
  of these embeddings $\Phi$ satisfying the disjoint union
  $\Phi \sqcup \rho \circ \Phi = \Hom(K, \Omega)$ is called a CM-type
  on $K$.
\end{defn}

\subsection{Injectivity of the reduction map}
\begin{defn} Let $A$ be an abelian variety over a field $k$ with
  complex multiplication by the maximal order $\cO_K$ in a CM-field
  $K$, and let $\fa$ be an ideal in $\cO_K$. A surjective homomorphism
  $\lambda_{\fa}: A \rightarrow A^{\fa}$, to an abelian variety
  $A^{\fa}$, is an $\fa$-multiplication if every homomorphism
  $a: A \rightarrow A$ with $a \in \fa$ factors through
  $\lambda_{\fa}$, and $\lambda_{\fa}$ is universal for this property,
  in the sense that, for every surjective homomorphism
  $\lambda':A \rightarrow A'$ with the same property; there is a
  homomorphism $\alpha: A' \rightarrow A^{\fa}$, necessarily unique,
  such that $\alpha \circ \lambda' = \lambda_{\fa}$.

\end{defn}


For abelian varieties $A$ and $B$ defined over a number field and with good
reduction modulo a prime $\fP$, the next proposition gives a condition
under which $A$ and $B$ will be isomorphic provided that their
reductions modulo $\fP$ are isomorphic.  The fact that the conditions
below are sufficient for an isomorphism to lift was given for
dimension 2 in \cite[Theorem 2]{EL09}. Here we give a general proof of
this fact.

\begin{prop} Let $(A, \iota)$, $(B, \iota')$ be simple, abelian
  varieties of type $(K, \Phi)$ defined over a number field
  $k$. Furthermore, assume that $\fP$ is a prime of $k$ such that $A$
  and $B$ have good reduction modulo $\fP$ and denote by $\tilde{A}$
  and $\tilde{B}$ their reductions modulo $\fP$ respectively. If
  $\tilde{A}$ and $\tilde{B}$ are simple with endomorphism ring
  isomorphic to $\cO_K$ and $\gamma: \tilde{A} \rightarrow \tilde{B}$
  is an isomorphism over $\overline{\FF}_p$, then $A$ and $B$ are
  isomorphic over $\overline{k}$.
\label{thm:isomorphism-lifts}
\end{prop}

\begin{proof}
  As $(A, \iota), (B, \iota')$ have the same type then, by
  \cite[Chapter II, Proposition 16]{Shi98}, they are isogenous via an $\fa$-multiplication, which we denote by $\lambda_{\fa}$. After possibly
  taking a field extension and picking a prime above $\fP$, we can
  assume that $\lambda_{\fa}$ and all endomorphisms are defined over
  $k$. The reduction $\tilde{\lambda_{\fa}}$ is also an
  $\fa$-multiplication \cite[Proposition 7.30]{Mil06}. Define an
  embedding $\tilde{\iota} : \cO_K \rightarrow \End(\tilde{A})$ by
  $\tilde{\iota}(a) = \widetilde{\iota(a)}$. This map is an
  isomorphism. Let $a \in \cO_K$ be such that
  $\tilde{\iota}(a) = \gamma^{-1} \circ \tilde{\lambda_{\fa}}
  \in \End(\tilde{A})$. As $\tilde{\iota}(a)$ factors through
  $\tilde{\lambda_{\fa}}$, $a \in \fa$ by \cite[Corollary
  7.24]{Mil06}. Also, $\iota(a)$ must factor through the
  $\fa$-multiplication, $\lambda_\fa$, that is,
  $\iota(a) = \gamma_1 \circ \lambda_\fa$ for $\gamma_1$ some isogeny
  from $B$ to $A$.

  Reducing modulo $\fP$,
  $\tilde{\iota}(a) = \tilde{\gamma_1} \circ \tilde{\lambda_\fa}$. As
  $\lambda_\fa$ is surjective, this implies
  $\gamma^{-1} = \tilde{\gamma_1}$. Similarly, we can find a
  $\gamma_2$ such that $\tilde{\gamma_2} = \gamma$. Then
  $\tilde{\gamma_1} \circ \tilde{\gamma_2} = \gamma^{-1} \circ \gamma
  = id$. As the reduction map is injective,
  $\gamma_1 \circ \gamma_2 = id$ and $\gamma_2 \circ \gamma_1 = id$,
  thus $A$ and $B$ are isomorphic.
\end{proof}

\subsection{The congruence relation}
\label{sec:cong-rel}
Let $(A,\iota)/\CC$ be of type $(K, \Phi)$ with $\End(A) \cong
\cO_K$. Denote by $(K^*, \Phi^*)$ the reflex of $(K, \Phi)$. Let $k$
be a field of definition for $(A,\iota)$. As the Hilbert class field
$H$ of $K^*$ is a field of definition for $(A,\iota)$ (see
\cite[Proposition 2.1]{Gor97}), we may assume that $k \subseteq
H$. Take $L$ to be a Galois extension of $\QQ$ containing the field of
definition $k$ and the field $K$. Recall $k$ contains $K^*$ by
\cite[Chapter III, Theorem 1.1]{Lan83}. Let $\fP$ be a prime of $k$ at
which $A$ has good reduction. Let $\fP_{K^*}$ be the prime of $K^*$
below $\fP$. Pick a prime $\fP_L$ of $L$ above $\fP$ and write
$\Phi_L^{-1}$ for the set of elements $\psi$ of $\Gal(L/\QQ)$ such
that $(\psi^{-1})_{|K} \in \Phi$.

Let $\pi \in \cO_K$ be such that $\tilde{\iota}(\pi)$ is the
N$_{k/\QQ}(\fP)$-th power Frobenius on the reduction $\tilde{A}$. In
Section~\ref{sec:isomorphism-classes} we will use the following
proposition, which is an easy consequence of the Shimura-Taniyama
congruence relation, to obtain a bijection between abelian varieties
with CM by $\cO_K$ of type $\Phi$ and abelian varieties over a finite
field satisfying certain properties.

\begin{prop} Assume that $p$ splits completely in $K$ and splits completely into principal ideals
  in $K^*$. Also, let $M$ be the
  Galois closure of the compositum of $K$ and $K^*$ and let $\fP_M$ be
  a prime above $\fP_{K^*}$. Write $\Phi_M^{-1}$ for the set of
  elements $\gamma$ of $\Gal(M/\QQ)$ such that $(\gamma^{-1})_{|K} \in
  \Phi$. Then $\pi\cO_M = \prod_{\gamma \in \Phi_M^{-1}}(\fP_M)^{\gamma}$.
\label{thm:frob-cond}
\end{prop}
\begin{proof} As $p$ splits completely into principal ideals in $K^*$,
  $p$ splits completely in the Hilbert class field $H$ of $K^*$. Thus,
  as mentioned above, $p$ splits completely in the field of definition
  $k$. Therefore, $f(\fP_L/\fP) = 1$, and by \cite[Chapter 3, Thm
  3.3]{Lan83} we obtain
  \[\pi\cO_L = \prod_{\psi \in \Phi_L^{-1}}\fP_L^{\psi}\cO_L.\] Using
  the splitting conditions on $p$ and intersecting with $\cO_M$ on
  both sides, we get the desired result.
\end{proof}

Thus the CM-type determines the ideal generated by Frobenius. We will
also need a version of this statement over $\QQ_p$.
Fix an algebraic closure $\overline{\QQ}_p$ of $\QQ_p$. Let 
\[H_w = \{\phi \in \Hom(K, \overline{\QQ}_p)  : \ \phi \text{ factors
  through } K \rightarrow K_w\},\]
where $K_w$ is the completion of $K$ at the place $w$.


\begin{prop} Let $(A, \iota)$ be an abelian variety with CM by the
  full ring of integers $\cO_K$ and of CM-type $\Gamma$. Moreover,
  assume $(A, \iota)$ has a model over the $p$-adic integers
  $\ZZ_p$. If $p$ splits completely in $K$,
  $\Gamma = \{\phi : \phi \in H_v, \text{ where } v \mid \pi \cO_K\}.$
\label{thm:Frob-determines-type}
\end{prop}
\begin{proof}
  By \cite[Lemme 5]{Tat71}, $\dfrac{v(\pi)}{v(q)} =
  \dfrac{\textsl{Card}(\Gamma \cap H_v)}{[K_v : \QQ_p]}$. If $p$ splits
  completely in $K$, then $[K_v : \QQ_p] = 1$ for all $v \mid p$ and $q = p$. This gives
  $v(\pi) = \textsl{Card}(\Gamma \cap H_v).$

  Also, as $p$ splits completely in $K$, there is only one embedding
  $K \rightarrow K_v$ for every $v \mid p$. Thus $\textsl{Card}(H_v) = 0$
  or $1$, and $\textsl{Card}(\Gamma \cap H_v) = 1$ if and only if
  $v(\pi)=1$.
\end{proof}

\section{Invariants of Picard curves}
\label{sec:invariants}
In this section, we discuss invariants for Picard
curves. Recall, if $y^3 = f(x)$ where $\deg(f) = 4$ and $f$ has no repeated
roots over the algebraic closure, then this defines a smooth curve
known as a {\em Picard curve}. Assume $L$ is a field of characteristic not 2
or 3, and let $C$ be a Picard curve over $L$. We can express the curve
$C$ in the form $y^3 = x^4 + g_2x^2 + g_3x + g_4$. This is called the {\em normal form} of the curve \cite[Appendix 1,
Definition 7.6]{Hol95}.

As in \cite[Section 1]{KLS18}, we define the
following three invariants for a Picard curve in normal form as
$j_1 := g_2^3/g_3^2, j_2 := g_2g_4/g_3^2, j_3 := g_4^3/g_3^4$.

We can write down a model for the curve with given invariants as follows:

{\it Case 1:} If $j_1 \ne 0$, then $C: y^3 = x^4 + j_1x^2 + j_1x + j_1j_2$.

{\it Case 2:} If $j_1 = 0, j_3 \ne 0$, then $C: y^3 = x^4 + j_3^2x + j_3^3$.

{\it Case 3:} If $j_1 = 0, j_2 = 0, j_3 = 0$, then $C: y^3 = x^4 + x$.

If $g_3 = 0$, then $C$ is a double cover of an elliptic curve (see
\cite[Lemma 2.1]{KLS18} and \cite[Theorem 2.4]{KLS18}). Thus the
invariants for a Picard curve $C$ whose Jacobian is simple are always
defined. This gives us the following proposition.

\begin{prop} Let $C$ be a Picard curve over a field $L$ of
  characteristic not 2 or 3 with $\Jac(C)$ simple. Assume that the
  three invariants $j_i(C)$ are defined over a subfield $k$ of
  $L$. Then $C$ has a model as a Picard curve over $k$.
\label{thm:model-where-invs-defined}
\end{prop}

Goren and Lauter showed that for genus 2 curves which have CM by a
given primitive, quartic, CM-field $K$ one can bound the primes
occurring in the denominators of the Igusa class polynomials in terms of a
value depending on $K$ \cite{GL07}. They obtain this bound by relating
the primes occurring in the denominators to primes of bad reduction of
the curves. For genus 3 curves with CM by a sextic CM-field $K$, a
bound on the primes of bad reduction in terms of a value depending on
$K$ was obtained in \cite{BCLGMNO15} and \cite{KLLNOS16}. A bound on
the primes occurring in the denominators of the above invariants of
Picard curves was obtained in \cite{KLS18}. 

We will need the following condition for Picard curves.

\begin{prop} Let $K=\QQ(\mu)$ be a sextic CM-field, $\Phi$ a primitive
  CM-type on $K$ and $p$ be a rational prime that splits completely in
  $K$. Let $C$ be a genus 3 curve defined over a number field $M$ with
  CM by the maximal order $\cO_K$ of $K$ and with type $\Phi$. Let $\fP$
  be a prime of $M$ above $p$. Then $C$ has potential good reduction
  at $\fP$. Moreover, if $C$ is a Picard curve then
  $v_\fP(j_i(C)) \ge 0$ for all invariants $j_i$.
\label{thm:bad-red-cond}
\end{prop}
\begin{proof}
  Assume $C$ has geometrically bad reduction modulo a prime $\fP$ of
  $M$ above the rational prime $p$. After possibly extending $M$, we
  may assume that $C$ has a stable model over $M$ and $\Jac(C)$ has
  good reduction over $M$. The stable reduction $\tilde{C}$ has at
  least 2 irreducible components \cite[Proposition
  4.2]{BCLGMNO15}. $\widetilde{\Jac(C)}$ is isomorphic as a polarized
  abelian variety to the product of the Jacobians of the irreducible
  components of $\tilde{C}$. That is, $\widetilde{\Jac(C)}$ is
  isomorphic as a principally polarized abelian variety to
  $E \times A$ \cite[Corollary 4.3]{BCLGMNO15}, where $E$ is an
  elliptic curve and $A$ is a two-dimensional principally polarized
  abelian variety. However, as $p$ splits completely in $K$, the
  reduction modulo $\fP$ of $\Jac(C)$ must be simple with CM by $K$ by
  \cite[Chapter 3, Theorem 2]{Shi98}. By~\cite[Theorem 1.2]{Sug14}s
  $\widetilde{\Jac(C)}$ is ordinary, so
  $\End(\widetilde{\Jac(C)}) \otimes \QQ$ is unchanged after base
  extension by \cite[Theorem 7.2]{Wat69}. Therefore
  $\widetilde{\Jac(C)}$ is geometrically simple as the endomorphism
  ring tensored with $\QQ$ is a field. This is a contradiction,
  so $C$ must have potential good reduction.

  Now assume that $C$ is a Picard curve and that $v_\fP(j_i(C)) < 0$
  for some $j_i$. After possibly extending $M$, we may assume that
  $\Jac(C)$ has good reduction modulo $\fP$. Then the reduction of
  $\Jac(C)$ modulo $\fP$ has two non-trivial
  abelian subvarieties by \cite[Lemma 2.1]{KLS18}. However, as $p$
  splits completely in $K$, we again obtain a contradiction.
\end{proof}
\begin{remark} It was pointed out to the authors by some of the
  anonymous referees and by Marco Streng that a similar condition to
  the above proposition was given in \cite[Proposition 4.1]{KLLRSS17} when
  the field $K/\QQ$ is cyclic Galois.
\end{remark}

\begin{remark}\label{sec:enum-isom-classes} To generate representatives for all distinct isomorphism classes, we
  use the invariants described in \cite[Section 4]{KW05}. To see that
  this enumerates all isomorphism classes of Picard curves with no
  repetitions see \cite[Appendix 1, Section 7.5]{Hol95}.
\end{remark}

\section{Reduction of class polynomials}
\label{sec:red-of-class}

Fix a sextic CM-field $K$ containing the cube roots of unity and a primitive CM-type $\Phi$ on $K$.
In the introduction we defined class polynomials $H_i^\Phi$ for $i=1,...,3$,
\[H_i^\Phi := \prod (X - j_i(C)),\]
where the product runs over all isomorphism classes of Picard curves defined over $\CC$ whose
Jacobian has complex multiplication by $\cO_K$ and of type
$\sigma \Phi$ for some $\sigma \in \Gal(\overline{\QQ}/\QQ)$.
\begin{remark} 
If one wants to use the class polynomials above to construct Picard
curves over $\CC$ with $\End(\Jac(C)) \cong \cO_K$, then one
needs to match up the roots of the three polynomials to obtain a triple
of roots $(j_1,j_2,j_3) $ that corresponds to such a curve. 
In genus 2, alternate class polynomials were proposed based on
Lagrange interpolation that prescribe which roots of the second and
third Igusa class polynomials to choose once the first has been
chosen\cite[Section 3]{GHKRW06}. These polynomials only work if the
first Igusa class polynomial has simple roots. For a discussion of
resolving this issue in genus 2 see \cite[Chapter III, Section
5]{Str10}.
\end{remark}

We will show that under suitable restrictions on the prime $p$, the
reduction modulo $p$ of these polynomials $H_i^\Phi$ is
\[H_{i,p} := \prod (X - j_i(C)),\]
where the product runs over all $\overline{\FF}_p$-isomorphism classes
of Picard curves $C$ which arise as the reduction of Picard curves over
$\CC$ that have complex multiplication by $\cO_K$ and type
$\sigma \Phi$ for some $\sigma \in \Gal(\overline{\QQ}/\QQ)$.

First we describe when a principally polarized abelian variety is the
Jacobian of a Picard curve.  

In the following, whenever we assume that
a field $F$ contains the cube roots of unity, it is also implied that
$F$ does not have characteristic 3.

\begin{lemma} Let $(A, \mathcal{C})$ be a simple, principally polarized
  abelian variety of dimension 3 over a perfect field $H$ which
  contains the cube roots of unity. In addition, assume $(A,\mathcal{C})$
  has complex multiplication by $K$ with $\QQ(\zeta_3) \subset
  K$. Then $(A, \mathcal{C})$ is geometrically the Jacobian of a Picard
  curve $C$ which has a model over $H$.
\label{thm:pic-model-hcf}
\end{lemma}
\begin{proof}
  By \cite[Lemma 1]{KW05},
  $(A, \mathcal{C})$ is the Jacobian of a Picard curve $C$ after we
  base change to a finite extension $L$ of $H$.  After possibly
  another finite extension, we may assume $L$ is Galois over $H$. Let
  $\sigma \in \Gal(L/H)$, then
  $\Jac(C^\sigma) \cong_L \Jac(C)^\sigma$. As $\Jac(C)$ has a model
  over $H$, $\Jac(C^\sigma) \cong_L \Jac(C)$.

  Hence by Torelli's theorem, $C \cong_{L} C^\sigma$. So
  $j_i(C) = j_i(C^\sigma) = j_i(C)^\sigma$, $i=1,...,3$. Therefore the
  invariants $j_i(C)$ are defined over $H$. As the invariants $j_i(C)$
  are defined over $H$, Proposition~\ref{thm:model-where-invs-defined}
  implies that $C$ has a model over $H$.
\end{proof}

Before we discuss reductions of our class polynomials, we need the following. 


\begin{prop}
$H_1^\Phi,H_2^\Phi,H_3^\Phi$ are polynomials defined over $\QQ$.
\end{prop}
\begin{proof}
  Every abelian variety with CM by $K$ has a model over a number
  field. Thus, by \cite[Theorem 4]{OU73}, the curve $C$ is also
  defined over a number field. So if
  $\sigma \in \Gal(\overline{\QQ}/\QQ)$ is an automorphism, then the
  tuple of invariants $j_i(C)^\sigma$ corresponds to the curve
  $C^\sigma$. But if $\Jac(C)$ has CM-type $(K, \Phi)$ under some
  embedding $\iota: K \hookrightarrow \End(\Jac(C)) \otimes \QQ$, then
  $\Jac(C^\sigma)$ has CM-type $(K, \sigma \Phi)$ by
  \cite[Chapter 3, Theorem 1.2]{Lan83}. The number of roots of the $H_i^\Phi$ is finite as
  there are only finitely many principally polarized abelian varieties
  with endomorphism ring isomorphic to $\cO_K$ of type $\sigma\Phi$ \cite[Chapter 3, Corollary
  2.7]{Lan83}, so the $H_i^\Phi$ are polynomials defined over $\QQ$.
\end{proof}

%
%

We will use the abbreviation p.p.a.v.\ for a principally polarized
abelian variety. For a CM-field $K$ of degree $2g$ over $\QQ$,
let
\[\Isom{K}{\Phi} = \{ \CC\text{-isomorphism classes of simple p.p.a.v.}
  \text{ with CM by } \cO_K \text{ of type } \Phi\}.\]\label{defn:isom-k-phi} 
The abelian varieties in this set are of dimension $g$.
By \cite[Proposition 2.1]{Gor97}, every p.p.a.v.\ $(A, \cC)$
representing an isomorphism class in $\Isom{K}{\Phi}$ has a model over
the Hilbert class field $H$ of the reflex field $K^*$ which has good
reduction modulo any prime $\fP$ of $H$. By~\cite [Chapter II,
Proposition 6.7]{Mil06}, the reduction of the polarization $\cC$ is a
polarization on the reduced variety $\tilde{A}$. If $p$ splits
completely into principal ideals in $K^*$ then $p$ splits completely
into principal ideals in $H$. Thus, the reduction $(A_\fP, \cC_\fP)$
of $(A, \cC)$ modulo $\fP$ has a model over $\FF_p$. Denote by
$\Isomr{K}{\Phi}$ the set of $\overline{\FF}_p$-isomorphism classes
occurring in this way. That is,
\[\Isomr{K}{\Phi} =
\{\overline{\FF}_p\text{-isomorphism classes of p.p.a.v.'s } (A_\fP, \cC_\fP)/\FF_p \mid (A, \cC) \in \Isom{K}{\Phi} \}.
\]
\label{defn:isomr-k-phi} 
\begin{prop} Let $\sigma \in \Gal(\overline{\QQ}/\QQ)$. If
  $\Phi\gamma = \sigma\Phi$ for some $\gamma \in Aut(K/\QQ)$, then
  $\Isom{K}{\Phi}$ and $\Isom{K}{\sigma\Phi}$ are equal. Otherwise,
  $\Isom{K}{\Phi}$ and $\Isom{K}{\sigma\Phi}$ are disjoint.
	\label{thm:inequivalent-types}
\end{prop}
\begin{proof}
	For the first statement see \cite[Pg 22]{Str10}. The second statement follows from \cite[Chapter I, Lemma 5.6]{Str10}. 
\end{proof}

For a sextic CM-field $K$ containing the cube roots of unity, define:
\[\Pic{} := \{\text{Picard curves } C \text{ over } \CC \text{ } |
\text{ } \Jac(C) \in \Isom{K}{\Phi}\}/\text{isomorphism over } \CC,\]
and
\[\Picr{} := \{\text{Picard curves } C \text{ over } \FF_p \mid \Jac(C) \in \Isomr{K}{\Phi}\}/\text{isomorphism over } \overline{\FF}_p.\]
Let $p > 3$ be a rational prime that splits completely in $K$ and splits
completely into principal ideals in $K^*$.

\begin{prop}
The reduction of the polynomials $H_i^\Phi$ modulo a prime satisfying the above conditions gives
$H_i^\Phi \mod p \equiv \prod(X-j_i(C))$,
where the product is over all $C$ such that $C$ is in $\Picr{\sigma}$ for some $\sigma \in \Gal(\overline{\QQ}/\QQ)$.
\label{thm:class-polys-reduce}
\end{prop}
\begin{proof}
  As $p$ splits completely into principal ideals in $K^*$, the reflex
  field for $(K, \Phi)$, it splits completely in $H$. Let $\fP$ be a
  prime of $H$ above $p$. By \cite [Proposition 2.1]{Gor97}, $\Jac(C)$
  is defined over $H$ for any curve $C$ in $\Pic{}$. Then $C$ itself
  also has a model over $H$ by
  Proposition~\ref{thm:pic-model-hcf}. $C$ has potential good
  reduction by Proposition~\ref{thm:bad-red-cond}, so let $L$ be a
  finite extension over which $C$ obtains good reduction. Furthermore,
  let $\fP_L$ be a prime above $\fP$. Thus, the reduction $C_{\fP_L}$
  of $C$ modulo $\fP_L$ will be defined over possibly a finite
  extension of $\FF_p$. However, as the invariants of $C$ belong to
  $H$, the invariants of $C_{\fP_L}$ belong to $\FF_p$ so $C_{\fP_L}$
  has a model over $\FF_p$. Thus, we get a map from $\Pic{}$ to
  $\Picr{}$. For any $\sigma \in \Gal(\overline{\QQ}/\QQ)$, let
  $K_\sigma^*$ be the reflex field for the type $(K,\sigma\Phi)$. One
  can check that the reflex fields $K^*$ and $K_\sigma^*$ are
  isomorphic over $\QQ$ . Therefore, $p$ splits completely into
  principal ideals in the reflex field of $K_\sigma^*$, so we also get
  a map from $\Pic{\sigma}$ to $\Picr{\sigma}$ induced by reduction
  modulo $\fP_L$. It remains to show that the reduction map induces a
  bijection. Taking Jacobians of elements in $\Pic{}$ and $\Picr{}$
  gives bijective maps into $\Isom{K}{\Phi}$ and $\Isomr{K}{\Phi}$
  respectively.


%

  The map $\Isom{K}{\Phi}$ to $\Isomr{K}{\Phi}$ induced by reduction
  modulo $\fP$ is injective by Proposition~\ref{thm:inj-isom-classes}. By definition, the map from
  $\Isom{K}{\Phi}$ to $\Isomr{K}{\Phi}$ is surjective, so it follows
  that $\Pic{}$ is in bijection with the set $\Picr{}$ under the
  reduction map. The sets $\Isom{K}{\Phi}$ and $\Isom{K}{\sigma\Phi}$
  are either equal or distinct by
  Proposition~\ref{thm:inequivalent-types}. The elements in
  $\Isomr{K}{\Phi}$ are simple with CM by $\cO_K$ by Proposition~\ref{thm:has-cm}. Thus, the sets $\Isomr{K}{\Phi}$ and
  $\Isomr{K}{\sigma\Phi}$ are equal if and only if $\Isom{K}{\Phi}$
  and $\Isom{K}{\sigma\Phi}$ are equal by
  Proposition~\ref{thm:isomorphism-lifts}. Therefore, bijectivity of
  the map from $\Pic{}$ to $\Picr{}$ suffices to prove the
  proposition.
\end{proof}

\section{Computing $H_i^\Phi$ modulo $p$}
\label{sec:isomorphism-classes}

Let $(K,\Phi)$ be a primitive CM-type. Denote by $(K^*, \Phi^*)$ the
reflex of $(K, \Phi)$. Let $H$ be the Hilbert class field of $K^*$ and
$M$ the normal closure of the compositum of $K$ and $K^*$. Let $L$ be the Galois closure of the compositum of $H$ and $M$ over $\QQ$. Take $p$ to
be a rational prime which splits completely into principal ideals in
$K^*$ and splits completely in $K$. Denote by $\fP$ a prime of $H$
above $p$, $\fP_L$ a prime of $L$ above $\fP$ and $\fP_M$ a prime of
$M$ below $\fP_L$. Denote by $\Phi_M^{-1}$ the set of elements
$\psi_i$ of $\Gal(M/\QQ)$ such that $(\psi_i^{-1})_{|K} \in \Phi$.

\subsection{An equivalent definition of $\Isomr{K}{\Phi}$}

In this subsection, we give an equivalent definition of
$\Isomr{K}{\Phi}$ in terms of a condition on the Frobenius of the
abelian varieties in $\Isomr{K}{\Phi}$. This new definition is more
suitable for computations. In particular, we will use it in
computing the set $\Picr{}$ which occurs in the description of the
class polynomials $H_i^\Phi$ modulo $p$ in Theorem~\ref{thm:class-polys-reduce}.
For a CM-field $K$ with $[K:\QQ] = 2g$, recall the definitions of $\Isom{K}{\Phi}$ and $\Isomr{K}{\Phi}$ from Section~\ref{sec:red-of-class}.

%

We will now define a set $\altdef{K}{\Phi}$ which we will show is
equal to the set $\Isomr{K}{\Phi}$. The main tool that allows us to
give this equivalent description will be the Shimura-Taniyama
Congruence relation, specifically the statement in
Proposition~\ref{thm:frob-cond}, which relates the CM-type of an
abelian variety defined over a number field with CM to the ideal
generated by Frobenius of the reduction of the abelian variety modulo
$\fP$. In genus 2, this idea was used in \cite{LR13} to describe the
set we refer to as $\Isomr{K}{\Phi}$.

With notation as above, denote by $\altdef{K}{\Phi}$ the set of all
$\overline{\FF}_p$-isomorphism classes of ordinary, simple,
principally polarized abelian varieties $(A, \cC)$ of dimension $g$ defined over $\FF_p$
with CM by $\cO_K$ satisfying the following
condition: For $(A, \cC)$ a representative of an $\overline{\FF}_p$
class as above, there exists an embedding $\iota$ of $K
\hookrightarrow \End(A) \otimes \QQ$ such that, under this embedding,
the element $\pi$ for which $\iota(\pi)$ is the Frobenius endomorphism
on $A$ satisfies
\begin{equation}\label{eqn:frob-cond}
 \pi\cO_M = \prod_{\phi \in \Phi_M^{-1}}\fP_M^{\phi}. 
\end{equation}

Recall, in the beginning of the section, we fixed a prime $\fP_L$ of $L$ above the prime $\fP$ of $H$ and define $\fP_M = \fP_L \cap M$. One can easily check that $\Isomr{K}{\Phi}$ does not depend on the choice of $\fP_L$ above $\fP$. We now wish to show that the sets $\altdef{K}{\Phi}$ and $\Isomr{K}{\Phi}$ are equal. First we show

\begin{prop} Every element in $\Isomr{K}{\Phi}$ is ordinary, geometrically simple with endomorphism ring isomorphic to $\cO_K$.
\label{thm:has-cm}
\end{prop}
\begin{proof}
  Let $(A, \cC)$ be a representative of a class in $\Isom{K}{\Phi}$
  such that it has good reduction modulo $\fP$ as above. Let $A_{\fP}$
  be the reduction of $A$ modulo $\fP$.  The reduction map gives an
  inclusion $\End(A) \hookrightarrow \End(A_\fP)$ \cite[Theorem
  3.2]{Lan83}, thus, $\cO_K$ embeds into $\End(A_\fP)$. By
  \cite[Chapter 3, Theorem 2]{Shi98}, the abelian variety $A_{\fP}$ is
  simple and $\End(A_{\fP}) = \cO_K$. Also, $A_\fP$ is ordinary
  by~\cite[Theorem 1.2]{Sug14}. Thus, $\End(A_{\fP}) \otimes \QQ$ is
  unchanged after base extension by \cite[Theorem 7.2]{Wat69}. Hence
  $A_\fP$ is geometrically simple as the endomorphism ring tensored
  with $\QQ$ is a field.
\end{proof}

The following two results are a generalization to arbitrary dimension of the dimension 2 case treated in \cite[Theorem 2]{EL09}.

\begin{prop} The reduction map $\Isom{K}{\Phi} \rightarrow \Isomr{K}{\Phi}$ is injective.
\label{thm:inj-isom-classes}
\end{prop}
\begin{proof}
Every element in $\Isomr{K}{\Phi}$ is simple with CM by $\cO_K$ by Proposition~\ref{thm:has-cm}. Thus, the proposition follows from applying Proposition~\ref{thm:isomorphism-lifts}.
 \end{proof}
\begin{thm} With notation as above, the set $\Isomr{K}{\Phi}$ is equal to the set $\altdef{K}{\Phi}$.
\label{thm:bijection-of-isom-classes}
\end{thm}
\begin{proof}
  We first show that $\Isomr{K}{\Phi} \subset \altdef{K}{\Phi}$. Let
  $(A, \cC)$ be a representative of a class in $\Isomr{K}{\Phi}$. By
  Proposition~\ref{thm:has-cm}, $A$ is ordinary and geometrically simple with $\End(A) \cong \cO_K$. As we remarked above, $p$ splits completely into principal
  ideals in $K^*$, so  the Frobenius of $A$ satisfies Equation~\ref{eqn:frob-cond} by Proposition~\ref{thm:frob-cond}. Hence $\tilde{A} \in \altdef{K}{\Phi}$.
This shows $\Isomr{K}{\Phi} \subset \altdef{K}{\Phi}$. It remains to show the
reverse inclusion. 

To do this, we will show that the two sets have the
same cardinality. Both sets are finite as there are only finitely many
isomorphism classes of principally polarized abelian varieties defined
over $\FF_p$. We know from the previous proposition that $\Isom{K}{\Phi}
\rightarrow \Isomr{K}{\Phi}$ is an injection. Thus, we have the inequality
of cardinalities: $|\Isom{K}{\Phi}| \le |\Isomr{K}{\Phi}| \le |\altdef{K}{\Phi}|$.

It suffices to show $|\altdef{K}{\Phi}| \le
|\Isom{K}{\Phi}|$. Therefore, we will show that there is an injective
map from $\altdef{K}{\Phi}$ into $\Isom{K}{\Phi}$. We define the map
as follows: Let $(A_0, \mathcal{C}_0)$ be an abelian variety representing
a class in $\altdef{K}{\Phi}$. Since $A_0$ is ordinary, we can
consider its Serre-Tate canonical lift \cite[Pgs 172-173, Theorem
3.3]{Mes72}
to $\ZZ_p$ which we will call $(A, \mathcal{C})$ .

As $(A_0, \mathcal{C}_0) \in \altdef{K}{\Phi}$ we have $\pi\cO_M = \prod_{\phi_\alpha \in \Phi_M^{-1}}(\fP_M)^{\phi_\alpha}$. Let $\{\psi_w\}$ be the set of all embeddings of $M$ into
$\overline{\QQ}_p$ induced by completion at a prime $\fP_w$ for $\fP_w
\mid \pi \cO_M$. By Proposition~\ref{thm:Frob-determines-type}, the
embeddings induced by completion at primes occurring in the
decomposition of the ideal generated by $\pi$ give the CM-type of
$A$. Under some embedding $\rho: \QQ_p \hookrightarrow \CC$, we can
verify that $\rho(A)$ has type $(K, \sigma\Phi)$ for some
$\sigma \in \Gal(M/\QQ)$. By \cite[Theorem 7]{Yal66}, modifying 
$\rho$ by an automorphism of $\CC$, we can arrange that $\rho(A)$ has
CM-type $(K, \Phi)$. As the choice of $\rho$ does not depend on $A$,
this gives us the injection from $\altdef{K}{\Phi}$ to
$\Isom{K}{\Phi}$. Hence $\altdef{K}{\Phi} = \Isomr{K}{\Phi}$.
\end{proof}

\subsection{Correctness proof for the main algorithm}

We must now show that the Chinese Remainder Theorem may be used to reconstruct the class polynomials from sufficiently many of the $H_{i,p}$. This is accomplished by the following whose proof is identical to that of \cite[Theorem 3]{EL09}:

\begin{thm} Let $M$ be the least common multiple of the denominators of the class polynomials and let $N$ be the maximum absolute value of the coefficients of the class polynomials. Let $B = 2NM$. Then if $S$ is a set of primes satisfying the conditions in Theorem~\ref{thm:main-algorithm}, we can use the Chinese remainder theorem on the polynomials $\{H_{i,p}\}_{p \in S}$, $i$ from $1$ to $3$, to reconstruct the polynomials $H_i^\Phi$.
\label{thm:CRT}
\end{thm}

\begin{remark} A definition of class polynomials for Picard curves and
	a bound on the primes occurring in the denominators are given in
	\cite[Theorem 1.3]{KLLNOS16}, and the class polynomials we define
	divide them. In genus 2, bounds on the denominators of the Igusa
	class polynomials were obtained in \cite{GL12}.
\end{remark}

\begin{proof}[Proof of Theorem~\ref{thm:main-algorithm}] 
Using Theorem~\ref{thm:class-polys-reduce}, we see that $H_{i,p} := \prod(X - j_i(C))$, where the product runs over representatives for elements in $\Picr{\sigma}$ for all $\sigma \in \Gal(\overline{\QQ}/\QQ)$.
We can enumerate all $\overline{\FF}_p$ isomorphism classes of Picard
curves defined over $\FF_p$ using the invariants discussed in
Remark~\ref{sec:enum-isom-classes}. We can check whether a curve is in
$\Picr{}$ by checking whether $\Jac(C)$ is in $\altdef{K}{\Phi}$ by
Theorem~\ref{thm:bijection-of-isom-classes}. This involves checking
that $\Jac(C)$ has complex multiplication by $\cO_K$ which can be
accomplished using the algorithm of Section~\ref{sec:endo-comp}. We
then perform the CRT step using Theorem~\ref{thm:CRT}.
\end{proof}

\section{Endomorphism ring computation}
\label{sec:endo-comp}

The algorithm of Theorem~\ref{thm:main-algorithm} requires us to
check whether certain genus 3 curves $C$ have complex multiplication
by a sextic CM-field $K$. 
 An algorithm for checking whether the
Jacobian of an ordinary genus 2 curve (i.e.\ a curve whose Jacobian is
ordinary) has complex multiplication by the full ring of integers of a
primitive quartic CM-field $K$ was presented, under certain
restrictions on the field $K$, in \cite{EL09}. Improvements to this
algorithm were presented in \cite{FL08} and \cite{LR13}. We generalize
these methods to the genus 3 case.

\begin{thm}

The following algorithm takes as input a sextic CM-field $K$ and an ordinary genus 3 curve $C$ over a field $\FF_p$ where $p$ splits completely in $K$. The algorithm outputs {\bf true} if $\Jac(C)$ has endomorphism ring the full ring of integers $\cO_K$ and {\bf false} otherwise:
 
\begin{enumerate}
\item Compute a list of all possible characteristic polynomials of Frobenius for ordinary, simple, abelian varieties with complex multiplication by $K$. Output false if the characteristic polynomial of $\Jac(C)$ is not in this list.
\item Compute a basis for $\cO_K$. 
\item For each element $\alpha$ of the basis in the previous step, use Proposition~\ref{thm:is-endo} to determine if it is an endomorphism. If it is not, output false.
\item Output true. 
\end{enumerate}
\end{thm}

The values for Frobenius in Step i) satisfy $\pi\overline{\pi} = p$
with $\pi \in \cO_K$, i.e.\ $N_{K/K^+}(\pi) = p$ where $K^+$ is the maximal totally real subfield of $K$. This relative norm
equation can be used to find all such values of $\pi$. By the
Honda-Tate theorem, every such $\pi$ will arise as the Frobenius of
some abelian variety $A$ over $\FF_p$. If the characteristic
polynomial of $\pi$ is irreducible, then $A$ is simple and $\QQ(\pi)
\cong K$. If $p$ does not divide the middle coefficient of the characteristic polynomial of Frobenius, then $A$ is ordinary \cite[Definition 3.1]{How95}. By \cite [Pg 97, Exemple b]{Tat71}, the endomorphism ring of
$A$ is an order in $K$.

\subsection{Determining if an element is an endomorphism}
\label{sec:generating-set}

Our approach in this subsection follows closely that of \cite[Section
3]{FL08} and \cite[Section 4]{LR13} for genus 2. We discuss some
changes which are required for genus 3. To determine if
$\End(\Jac(C)) \cong \cO_K$, we wish to check, for some $\ZZ$-basis of
$\cO_K$, $\alpha_1, ..., \alpha_6$, whether each $\alpha_i$ is an
endomorphism. As $\ZZ[\pi]$ is an order in $K$, for every
$\alpha \in \cO_K$, we can write

\begin{align}
\alpha = P_\alpha(\pi)/n:=(a_0 + a_1 \pi + ... + a_5\pi^5)/n. \label{eqn:express-in-frob-full-denom}
\end{align}

for some integer $n$. The next proposition lets us check if $\alpha \in \cO_K$ is an endomorphism of $\Jac(C)$:


\begin{prop} Let $C$ be an ordinary curve of genus 3 over
  $\mathbb{F}_p$ with $\End(\Jac(C)) \otimes \mathbb{Q}=K$, and suppose $p$ splits completely
  in $K$. Let $\alpha  = P_\alpha(\pi)/n \in \cO_K$ with $n = \prod \ell_i^{e_i}$. Then $\alpha$ is an
  endomorphism of $\Jac(C)$ if and only if $P_\alpha(\pi)$ is zero on the
  $\ell_i^{e_i}$-torsion for $\ell_i \ne p$.
\label{thm:is-endo}
\end{prop}
\begin{proof} 

  By \cite[Lemma 3.2]{FL08}, it suffices to check that each
  $P_\alpha(\pi)/\ell_{i}^{d_{i}}$ is an endomorphism. If $\ell_i$ is
  coprime to $p$, then by \cite[Corollary 9]{EL09}, we can check
  whether $P_\alpha(\pi)/\ell_i^{d_i}$ is an endomorphism by
  determining if $P_\alpha(\pi)$ is zero on the
  $\ell_i^{d_i}$-torsion.

  It remains to handle the case where $\ell_i = p$. For a group $A$,
  denote the $p$-primary part of $A$ by $A_p$. Write
  $[\cO_K : \ZZ[\pi]] = [\cO_K : \ZZ[\pi, \overline{\pi}]] \cdot
  [\ZZ[\pi, \overline{\pi}] : \ZZ[\pi]].$ It is not hard to see that
  $[\ZZ[\pi, \overline{\pi}] : \ZZ[\pi]]$ is a power of $p$ (See \cite
  [Corollary 3.6]{FL08}.) As $p$ splits completely in $K$, one can
  show, $p \nmid [\cO_K : \ZZ[\pi, \overline{\pi}]]$, thus
  $|(\cO_K/\ZZ[\pi])_p| = |(\ZZ[\pi,
  \overline{\pi}]/\ZZ[\pi])_p|$. This follows from an argument similar
  to \cite[Proposition 3.7]{FL08}.

But this implies for any $\beta \in \cO_K$, if $p^k\beta \in \ZZ[\pi]$ then $\beta \in \ZZ[\pi, \overline{\pi}]$. Thus, any such element is an endomorphism.
\end{proof}

\subsection{Computing the $\ell^d$-torsion and arithmetic}

The algorithm of Couveignes \cite{Cou09} shows how to compute the
$\ell^d$-torsion. Couveignes' method works for a very general class of
curves. However, we instead use some algorithms specific to Picard
curves. For a Picard curve $C/k$, where $k$ is a finite field,
Couveignes' method requires the ability to choose random points in
$\Jac(C)(k)$. This is easy to do if we represent elements of
$\Jac(C)(k)$ as formal sums of points on $C$. However, to do
arithmetic on $\Jac(C)(k)$, it is easier to represent elements as
ideals in the affine coordinate ring of $C$. Thus, we need to be able
to switch between the two representations. First, we recall the following consequence of the Riemann-Roch theorem:

\begin{prop} For $C$ a Picard curve and $P_\infty$ the point at infinity for the affine model described above, for any degree 0 divisor $D$ there is a unique effective divisor $E$ of minimal degree $0 \le m \le 3$ such that $E - m P_\infty$ is equivalent to $D$.
\label{thm:uniq-rep-pts}
\end{prop}
\begin{proof}
As Picard curves are non-singular with a $k$-rational point, the proof follows from \cite [Theorem 1]{GPS02}.
\end{proof}

We will call such a unique divisor above the reduced representation
of $D$. So to find a random point in $\Jac(C)(k)$, we can just pick
at most 3 random points on $C$.

A reduced divisor $D$ for which all points in the effective part $E$
lie in the same $\Gal(\overline{k}/k)$-orbit will be called an {\bf
  irreducible} divisor. Every degree 0 divisor can be expressed as a
sum of irreducible divisors.

We can also represent points on $\Jac(C)$ as elements of a particular
class group. Denote by $R = k[x,y]/\langle y^3-f(x) \rangle$ the coordinate ring of
$C$. By \cite[Proposition 2]{GPS02}, $R$ is the integral closure of
$k[x]$ in $k(C)$.

Given an irreducible divisor $P$ we can associate to it a prime ideal
$\fP$ of $R$. We can extend this to a map $\rho$ from effective divisors to
ideals of $R$ as:
\[\rho\left(\sum n_iP_i\right) := \prod \fP_i^{n_i},\]
with the $P_i$ irreducible divisors and the $\fP_i$  the corresponding primes of $R$.

\begin{prop} For $C$ a Picard curve over $k$ and $R$ the coordinate
  ring of $C$ described above, the map $\rho$ induces an isomorphism
$\Jac(C)(k) \rightarrow \Cl(R)$,
where $\Cl(R)$ is the class group of $R$.
\end{prop}
\begin{proof}
This follows from applying \cite[Proposition 3]{GPS02}.
\end{proof}

We refer to the image of a reduced divisor under the map $\rho$ as a reduced ideal. 

\begin{prop} Given a reduced divisor $D$, there is an algorithm to
  find generators $u(x), w(x,y)$ for the ideal $\rho(D)$. Moreover,
  given an ideal $I$ of $R$ in the form
  $I = \langle u(x), w(x,y)\rangle$, we can compute $\rho^{-1}(I)$.
\end{prop}
\begin{proof}
  As a reduced divisor is a sum of irreducible divisors, it suffices
  to associate to an irreducible divisor $Q$ the corresponding prime
  ideal. We can associate a prime ideal $\fP$ in $R$ by first
  considering the polynomial $u = \prod (x - x_i)$, where the product
  is over all $x$-coordinates of points in $Q$. We then take a
  polynomial $w(x,y)$ such that the set of common roots of $u, w$ is
  exactly the set of points of $Q$. If the $x_i$ are all distinct, then
  we take the polynomial $w = y-v(x)$, where $v(x)$ is the polynomial
  interpolating the points in $Q$. If the roots of $u(x)$ are not
  distinct, then we can construct $w$ in a way similar to the
  interpolation polynomial. In the case where there are two distinct
  $x$-coordinates $x_1, x_2$, let $y_1$ and $y_2$ be polynomials whose
  roots are the $y$-coordinates corresponding to $x_1$ and $x_2$,
  respectively. Then
\[w(x,y) := \dfrac{x - x_2}{x_1 - x_2}y_1(y) + \dfrac{x - x_1}{x_2 -
      x_1}y_2(y).\] If there is only a single $x$-coordinate, then we
  can write $w(x,y) = \prod(y-y_i)$, where the $y_i$ are the
  $y$-coordinates in the Galois orbit. The corresponding prime ideal
  in $R$ is then the ideal generated by $u$ and $w$.

We will now show how to explicitly find the inverse of $\rho$. Let
$D = \prod \fP_i^{n_i}$ be the ideal decomposition of $D$. Write
$\fP_i = \langle u(x), w(x,y)\rangle$. We can find the set of common zeroes of
$\fP_i$ by finding all roots $x_n$ of $u(x)$ and all roots $y_{n,m}$
of $w(x_n,y)$. Then the divisor $(\fP_i)$ equals $\sum (x_n,y_{n,m})$. Thus
we have constructed the inverse of the map $\rho$ on a prime divisor
$\fP$. By linearity, we can explicitly find the inverse
of any reduced ideal $D$ .
\end{proof}
There are several algorithms which perform arithmetic on $\Jac(C)(k)$
using the representation of points on $\Jac(C)(k)$ as ideals in the
class group, for example, \cite{GPS02}, \cite{Ari03}. In particular,
we will use the algorithm of \cite{Ari03} for the examples we
compute. To add two elements $P$, $Q$ of $\Jac(C)(k)$, one multiplies
the corresponding ideals to get an ideal $D$. One then wishes to get a
reduced ideal $D'$, to have a unique representative for the point
$D$. The algorithm of \cite{Ari03} gives a function $g$ such that $D'
= D + (g)$. The function $g$ is necessary for the computation of the
Weil pairing in the algorithm of Couveignes for computing torsion.

\section{Examples}
\label{sec:examples}
All examples were run on a computer with 4 Intel Xeon quad-core
processors and 64 GB of RAM. 

Let $K = K^+(\zeta_3)$, where $K^+$ is obtained by adjoining
to $\QQ$ a root of $x^3 - x^2 - 2x + 1$. We can verify that $K$ is
Galois with Galois group $\ZZ/6\ZZ$ and choose a primitive CM-type on
$K$. All types on $K$ are equivalent, so our choice does not matter. We
count the expected degree of our class polynomials using \cite[Pg
112, Note 3]{Shi98}. This is equivalent to counting the number of
elements in the {\it polarized class group} (see \cite{Bis15}), for
which there is a function in the AVIsogenies package \cite{BCR10}. We find that the
degree of the class polynomials for $K$ as above is 1. The first four
primes satisfying the conditions of Theorem~\ref{thm:main-algorithm}
are $13, 43, 97, 127$.  For $p=127$, our algorithm took 7 hours and 9
minutes of clock time and found one Picard curve in $\Picr{}$, that
is, one Picard curve whose Jacobian is in $\altdef{K}{\Phi}$:
$y^3 = x^4 + 75x^2 + 37x + 103$.

The Picard curve $\CC$ with CM by $\cO_K$, for $K$
as above, was computed in \cite{KW05}. However, the authors could not
verify that the curve they produce has CM by $\cO_K$. Our output
agrees with the result of their paper reduced modulo
$127$. Furthermore, assuming the curve they compute is correct, we get
a bound as in Theorem~\ref{thm:CRT} for the denominators and size of
coefficients in the class polynomials $H_i^\Phi$. In particular,
$N = 2^{12}$ and $M = 7$ work for the values in Theorem
~\ref{thm:CRT}. Using these values, we can run the CRT algorithm of
~\ref{thm:main-algorithm} to construct the class polynomials
$H_i^\Phi$ defined over $\QQ$. The algorithm took 8 hours, 55 minutes
to run.  We only need to reduce modulo the 4 primes $13, 43, 97,
127$. Our result agrees with the result of \cite{KW05, LS16}. Thus, our
algorithm can compute the class polynomials $H_i^\Phi$ given that one can
compute the bound in Theorem ~\ref{thm:CRT}. If we compare the algorithms on the small example we computed
above, the algorithm in~\cite{LS16} performs much faster: it was able
to compute the class polynomials in seconds. However,  since there are no known
bounds, yet, on the denominators of the class polynomials, no
complexity analysis has been done for our algorithm or the algorithms in~\cite{KW05,
  LS16}. So it is not clear how they would compare asymptotically.

Now let $K = K^+(\zeta_3)$, where $K^+$ is the field obtained by
adjoining to $\QQ$ a root of $x^3 + x^2 - 3x - 1$. This field is
non-Galois, and the Galois group of the normal closure over $\QQ$ is
$S_3 \times \ZZ/2\ZZ$. We also pick a CM-type $\Phi$ on $K$. We
compute that we expect our class polynomials to have degree 3 using
the polarized class group. We pick $p = 67$, which satisfies the
conditions of Theorem ~\ref{thm:main-algorithm}.  Our algorithm ran in
2 hours and 23 minutes, and we got 3 Picard curves over $\FF_p$ whose
Jacobians lie in $\altdef{K}{\sigma\Phi}$ for some
$\sigma \in \Gal(\overline{\QQ}/\QQ)$:

$y^3 = x^4 + 8x^2 + 64x + 61, \quad y^3 = x^4 + 62x^2 + 25x + 6, \quad y^3 = x^4 + 54x + 54$.

\section*{Acknowledgments}
The authors would like to thank Yuri Zarhin
for helpful discussions.  We thank the anonymous referees for several
helpful suggestions. We thank Marco Streng for valuable feedback on
an earlier version of this paper and for pointing out additional
references.

\end{document}